\documentclass[11pt,a4paper]{article}
\usepackage{epsfig,amsfonts,amsgen,amsmath,amstext,amsbsy,amsopn,amsthm}
\usepackage{color}
\usepackage{ebezier,eepic}
\newtheorem{theorem}{Theorem}

\newtheorem{lemma}{Lemma}
\newtheorem{false statement}{False statement}

\theoremstyle{definition}

\newtheorem{claim}{Claim}
\newtheorem{subclaim}{Claim}[claim]

\newtheorem{remark}[claim]{Remark}

\newtheorem{corollary}{Corollary}
\newtheorem{problem}{Problem}
\newtheorem{subproblem}{Problem}[problem]
\newtheorem{case}{Case}

\newtheorem{subcase}{Case}[case]

\newcounter{mathitem}
\newenvironment{mathitem}
  {\begin{list}{{$(\roman{mathitem})$}}{
   \setcounter{mathitem}{0}
   \usecounter{mathitem}
   \setlength{\topsep}{0pt plus 2pt minus 0pt}
   \setlength{\parskip}{0pt plus 2pt minus 0pt}
   \setlength{\partopsep}{0pt plus 2pt minus 0pt}
   \setlength{\parsep}{0pt plus 2pt minus 0pt}
   \setlength{\leftmargin}{35pt}
   \setlength{\itemsep}{0pt plus 2pt minus 0pt}}}
  {\end{list}}

\begin{document}

\title{\bf\Large Ore- and Fan-type heavy subgraphs for Hamiltonicity of 2-connected graphs\thanks{Supported by NSFC
(No.~11271300) and the Doctorate Foundation of Northwestern Polytechnical University (cx201326).}}

\date{}

\author{Bo Ning and Shenggui Zhang\thanks{Corresponding author. E-mail address: sgzhang@nwpu.edu.cn (S. Zhang).}    \\[2mm]
\small Department of Applied Mathematics,
\small Northwestern Polytechnical University,\\
\small Xi'an, Shaanxi 710072, P.R.~China}

\maketitle

\begin{abstract}
Bedrossian characterized all pairs of forbidden subgraphs for a 2-connected graph to be Hamiltonian. Instead of forbidding some induced subgraphs, we relax the conditions for graphs to be Hamiltonian by restricting Ore- and Fan-type degree conditions on these induced subgraphs. Let $G$ be a graph on $n$ vertices and $H$ be an induced subgraph of $G$. $H$ is called \emph{o}-heavy if there are two nonadjacent vertices in $H$ with degree sum at least $n$, and is called $f$-heavy if for every two vertices $u,v\in V(H)$, $d_{H}(u,v)=2$ implies that $\max\{d(u),d(v)\}\geq n/2$. We say that $G$ is $H$-\emph{o}-heavy ($H$-\emph{f}-heavy) if every induced subgraph of $G$ isomorphic to $H$ is \emph{o}-heavy (\emph{f}-heavy). In this paper we characterize all connected graphs $R$ and $S$ other than $P_3$ such that every 2-connected  $R$-\emph{f}-heavy and $S$-\emph{f}-heavy ($R$-\emph{o}-heavy and $S$-\emph{f}-heavy, $R$-\emph{f}-heavy and $S$-free) graph is Hamiltonian. Our results extend several previous theorems on forbidden subgraph conditions and heavy subgraph conditions for Hamiltonicity of 2-connected graphs.
\medskip

\smallskip

\end{abstract}

\section{Introduction}

We use Bondy and Murty \cite{Bondy_Murty} for terminology and notation not defined here and consider finite simple graphs only.

Let $G$ be a graph. For a vertex $v$ and a subgraph $H$ of $G$, we use $N_H(v)$ to denote the set, and $d_H(v)$ the number, of neighbors of $v$ in $H$, respectively. We call $d_H(v)$ the \emph{degree} of $v$ in $H$. For $x,y\in V(G)$, an $(x,y)$-\emph{path} is a path connecting $x$ and $y$; the vertex $x$ will be called the {\em origin} and $y$ the {\em terminus} of the path. If $x,y\in V(H)$, the \emph{distance} between $x$ and $y$ in $H$, denoted $d_H(x,y)$, is the length of a shortest $(x,y)$-path in $H$.  If there is no danger of ambiguity, $N_G(v)$, $d_G(v)$ and $d_G(x,y)$ are abbreviated to $N(v)$, $d(v)$ and $d(x,y)$, respectively.

If a subgraph $G'$ of a graph $G$ contains all edges $xy\in E(G)$ with $x,y\in V(G')$, then $G'$ is called an \emph{induced subgraph} of $G$. For a given graph $H$, we say that $G$ is \emph{$H$-free} if $G$ does not contain an induced subgraph isomorphic to $H$. For a family $\mathcal{H}$ of graphs, $G$ is called \emph{$\mathcal{H}$-free} if $G$ is $H$-free for every $H\in\mathcal{H}$.

The bipartite graph $K_{1,3}$ is called the \emph{claw}, its (only) vertex of degree 3 is called its \emph{center} and the other vertices are its \emph{end vertices}.
In this paper, instead of $K_{1,3}$-free, we use the terminology claw-free.

Many graph theorists drew their attention to find forbidden subgraph conditions for a graph to be Hamiltonian. If a graph is 2-connected and $P_3$-free, then it is a complete graph, and hence it is Hamiltonian. In fact, Faudree and Gould \cite{Faudree_Gould} showed that $P_3$ is the only connected graph $S$ such that every 2-connected $S$-free graph is Hamiltonian. The case with pairs of forbidden subgraphs other than $P_3$ is much more interesting. Bedrossian \cite{Bedrossian} gave a complete characterization of all pairs of forbidden subgraphs that imply a 2-connected graph is Hamiltonian.

\begin{theorem}[Bedrossian \cite{Bedrossian}]\label{th1}
Let $R$ and $S$ be connected graphs other than $P_3$ and let $G$ be a 2-connected graph. Then $G$ being $\{R,S\}$-free implies $G$ is Hamiltonian if and only if (up to symmetry) $R=K_{1,3}$ and $S=P_4,P_5,P_6,C_3,Z_1,Z_2,B,N$ or $W$ (see Fig. 1).
\end{theorem}

\begin{center}
\begin{picture}(360,200)

\thicklines

\put(5,140){\multiput(20,30)(50,0){5}{\put(0,0){\circle*{6}}}
\put(20,30){\line(1,0){100}} \put(170,30){\line(1,0){50}}
\qbezier[4](120,30)(145,30)(170,30) \put(20,35){$v_1$}
\put(70,35){$v_2$} \put(120,35){$v_3$} \put(170,35){$v_{i-1}$}
\put(220,35){$v_i$} \put(115,10){$P_i$ (Path)}}

\put(265,130){\put(20,30){\circle*{6}} \put(70,30){\circle*{6}}
\put(45,55){\circle*{6}} \put(20,30){\line(1,0){50}}
\put(20,30){\line(1,1){25}} \put(70,30){\line(-1,1){25}}
\put(20,10){$C_3$ (Cycle)}}

\put(0,0){\put(20,30){\circle*{6}} \put(70,30){\circle*{6}}
\multiput(45,55)(0,25){4}{\put(0,0){\circle*{6}}}
\put(20,30){\line(1,0){50}} \put(20,30){\line(1,1){25}}
\put(70,30){\line(-1,1){25}} \put(45,55){\line(0,1){25}}
\put(45,105){\line(0,1){25}} \qbezier[4](45,80)(45,92.5)(45,105)
\put(50,80){$v_1$} \put(50,105){$v_{i-1}$} \put(50,130){$v_i$}
\put(40,10){$Z_i$}}

\put(90,0){\put(45,40){\circle*{6}} \put(45,40){\line(-1,1){25}}
\put(45,40){\line(1,1){25}} \put(20,65){\line(1,0){50}}
\multiput(20,65)(50,0){2}{\multiput(0,0)(0,30){2}{\put(0,0){\circle*{6}}}
\put(0,0){\line(0,1){30}}} \put(25,10){$B$ (Bull)}}

\put(180,0){\multiput(20,30)(50,0){2}{\multiput(0,0)(0,30){2}{\put(0,0){\circle*{6}}}
\put(0,0){\line(0,1){30}}}
\multiput(45,85)(0,30){2}{\put(0,0){\circle*{6}}}
\put(45,85){\line(0,1){30}} \put(20,60){\line(1,0){50}}
\put(20,60){\line(1,1){25}} \put(70,60){\line(-1,1){25}}
\put(5,50){$a_1$} \put(75,50){$a_3$} \put(30,85){$a_2$} \put(30,115){$b_2$}
\put(5,25){$b_1$} \put(76,25){$b_3$}
\put(25,10){$N$ (Net)}}

\put(270,0){\put(45,30){\circle*{6}} \put(20,55){\line(1,0){50}}
\put(45,30){\line(1,1){25}} \put(45,30){\line(-1,1){25}}
\multiput(20,55)(0,30){2}{\put(0,0){\circle*{6}}}
\multiput(70,55)(0,30){3}{\put(0,0){\circle*{6}}}
\put(20,55){\line(0,1){30}} \put(70,55){\line(0,1){60}}
\put(50,25){$a_3$} \put(15,40){$a_2$} \put(70,40){$a_1$}
\put(5,85){$b_2$} \put(55,85){$b_1$} \put(55,115){$c_1$}
\put(10,10){$W$ (Wounded)}}

\end{picture}

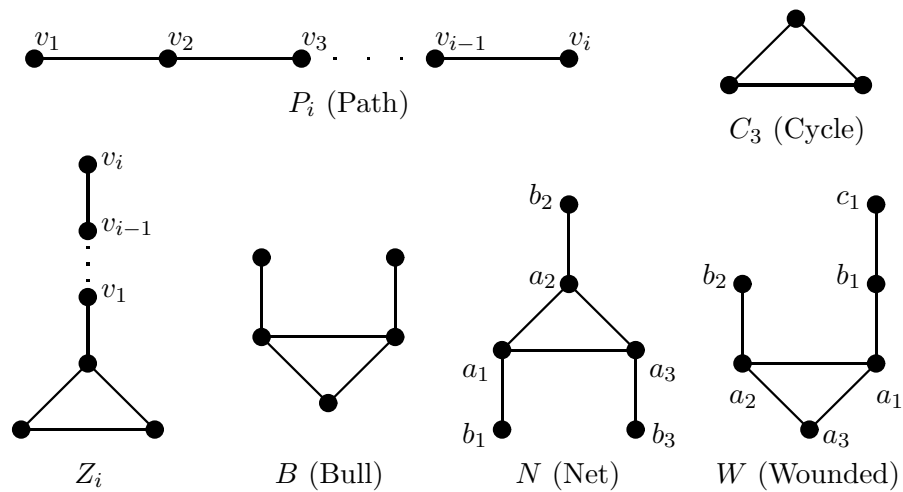
\small Fig. 1. Graphs $P_i,C_3,Z_i,B,N$ and $W$.
\end{center}

On the other hand, degree conditions have long been useful tools in the study of Hamilton cycles. Among all, Ore's condition \cite{Ore} is fundamental.

\begin{theorem}[Ore \cite{Ore}]\label{th2}
Let $G$ be a graph on $n\geq 3$ vertices. If the degree sum of every pair of nonadjacent vertices in $G$ is at least $n$, then $G$ is Hamiltonian.
\end{theorem}

Let $G$ be a graph on $n$ vertices. For a given graph $H$, we say that $G$ is \emph{$H$-o-heavy} if for every induced subgraph $G'$ of $G$ isomorphic to $H$, there exist two nonadjacent vertices $x,y\in V(G')$ such that $d(x)+d(y)\geq n$. For a family $\mathcal{H}$ of graphs, $G$ is called \emph{$\mathcal{H}$-o-heavy} if $G$ is $H$-\emph{o}-heavy for every $H\in \mathcal{H}$. Clearly, an $H$-free graph is also $H$-\emph{o}-heavy, and if $H'$ is an induced subgraph of $H$, then every $H'$-\emph{o}-heavy graph is also $H$-\emph{o}-heavy. In this paper, we use the terminology claw-\emph{o}-heavy instead of $K_{1,3}$-\emph{o}-heavy.

By relaxing forbidden subgraph conditions to conditions in which the subgraphs are allowed, but where Ore's condition is imposed on these subgraphs if they appear, Li et al. \cite{Li_Ryjacek_Wang_Zhang} extended Theorem \ref{th1} as follows.

\begin{theorem}[Li, Ryj\'a\v{c}ek, Wang and Zhang \cite{Li_Ryjacek_Wang_Zhang}]\label{th3}
Let $R$ and $S$ be connected graphs other than $P_3$ and let $G$ be a 2-connected graph. Then $G$ being $\{R,S\}$-$o$-heavy implies $G$ is Hamiltonian if and only if (up to symmetry) $R=K_{1,3}$ and $S=P_4,P_5,C_3,Z_1,Z_2,B,N$ or $W$.
\end{theorem}

One may notice that there is only one graph $P_6$ that appears in Bedrossian's result but misses here. Li et al. \cite{Li_Ryjacek_Wang_Zhang} also constructed a 2-connected claw-free $P_6$-\emph{o}-heavy graph which is not Hamiltonian. With a little effort, they got

\begin{theorem}[Li, Ryj\'a\v{c}ek, Wang and Zhang \cite{Li_Ryjacek_Wang_Zhang}]\label{th4}
Let $S$ be a connected graph other than $P_3$ and let $G$ be a 2-connected claw-$o$-heavy graph. Then $G$ being $S$-free implies $G$ is Hamiltonian if and only if $S=P_4,P_5,P_6,C_3,Z_1,Z_2,B,N$ or $W$.
\end{theorem}

There is another degree condition due to Fan \cite{Fan} (so-called Fan's condition) with respect to Hamilton cycles.

\begin{theorem}[Fan \cite{Fan}]\label{th5}
Let $G$ be a 2-connected graph on $n$ vertices. If $\max\{d(u),d(v)\}\geq n/2$ for every pair of vertices $u,v$ with $d(u,v)=2$, then $G$ is Hamiltonian.
\end{theorem}

Let $G$ be a graph on $n$ vertices. For a given graph $H$, we say that $G$ is \emph{$H$-f-heavy} if for every induced subgraph $G'$ of $G$ isomorphic to $H$, and two vertices $u,v\in V(G')$, $d_{G'}(u,v)=2$ implies that $max\{d(u),d(v)\}\geq n/2$. For a family $\mathcal{H}$ of graphs, $G$ is called \emph{$\mathcal{H}$-f-heavy} if $G$ is $H$-\emph{f}-heavy for every $H\in \mathcal{H}$. Note that an $H$-free graph is also $H$-\emph{f}-heavy. In contrast to the case of forbidden subgraphs or \emph{o}-heavy subgraphs, if $H'$ is an induced subgraph of $H$, then an $H'$-\emph{f}-heavy graph is not always $H$-\emph{f}-heavy. For example, $Z_2$ is an induced subgraph of $W$, but a $Z_2$-\emph{f}-heavy graph is not necessarily $W$-\emph{f}-heavy. As above, if $H=K_{1,3}$, then we use the terminology claw-\emph{f}-heavy instead of $K_{1,3}$-\emph{f}-heavy.

For a given graph $H\in \{P_4,P_5,P_6,Z_1,Z_2,B,N,W\}$, it is interesting to compare $H$-\emph{o}-heavy graphs with $H$-\emph{f}-heavy graphs. It is not difficult to see that there exist $H$-\emph{o}-heavy graphs which are not $H$-\emph{f}-heavy, and $H$-\emph{f}-heavy graphs which are not $H$-\emph{o}-heavy. Figure 2 shows a graph which is $N$-\emph{f}-heavy but not $N$-\emph{o}-heavy, $W$-\emph{o}-heavy and $W$-\emph{f}-heavy.

\begin{center}
\begin{picture}(200,150)(0,0)

\thicklines

\put(90,80){\put(45,30){\circle*{4}} \put(45,-30){\circle*{4}} \put(-50,20){\circle*{4}} \put(-90,20){\circle*{4}} \put(80,30){\circle*{4}} \put(80,-30){\circle*{4}} \put(-50,-20){\circle*{4}} \put(110,30){\circle*{4}} \put(-90,-20){\circle*{4}} \put(110,-30){\circle*{4}} \put(95,0){\circle*{4}}

\put(-50,20){\line(-1,0){40}} \qbezier(-50,20)(-2.5,25)(45,30) \qbezier(-50,20)(-2.5,-5)(45,-30) \put(45,-30){\line(0,1){60}} \put(45,-30){\line(1,0){35}} \put(45,30){\line(1,0){35}} \put(-50,-20){\line(-1,0){40}} \put(-90,-20){\line(0,1){40}} \put(80,-30){\line(1,0){30}} \put(80,-30){\line(1,2){30}} \put(80,30){\line(1,0){30}} \put(80,30){\line(1,-2){30}}

\put(0,0){\circle{120}} \put(-20,-30){$K_r$}
}

\end{picture}

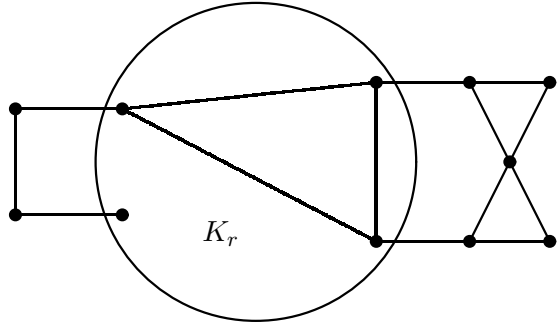
\small Fig. 2 A graph ($r\geq 7$) which is $N$-\emph{f}-heavy but not $N$-\emph{o}-heavy, $W$-\emph{o}-heavy and $W$-\emph{f}-heavy.
\end{center}

Our first aim in this paper is to find corresponding Fan-type heavy subgraph conditions which extend Theorem \ref{th1}. By Theorem \ref{th5}, we know that every 2-connected $P_3$-\emph{f}-heavy graph is Hamiltonian. Recall that $P_3$ is the only connected graph $S$ such that every 2-connected $S$-free graph is Hamiltonian and every $P_3$-free graph is $P_3$-\emph{f}-heavy. Thus $P_3$ is the only required graph $S$ such that every 2-connected $S$-\emph{f}-heavy graph is Hamiltonian. We have the following problem naturally.

\begin{problem}\label{pr1}
Which two connected graphs $R$ and $S$ other than $P_3$ imply that every 2-connected $\{R,S\}$-\emph{f}-heavy graph is Hamiltonian?
\end{problem}

By Theorem \ref{th1}, we get that (up to symmetry) $R=K_{1,3}$ and $S$ must be one of the graphs $P_4,P_5,P_6,C_3,Z_1,Z_2,B,N$ or $W$.

In fact, there are many previous results \cite{Bedrossian_Chen_Schelp,Chen_Wei_Zhang_0,Chen_Wei_Zhang_5,Li_Wei_Gao} which are related to Problem \ref{pr1}, although stated in different terminology and notations.

\begin{theorem}[Chen, Wei and Zhang \cite{Chen_Wei_Zhang_5}]\label{th6}
Let $G$ be a 2-connected graph. If $G$ is $\{K_{1,3},P_6\}$-$f$-heavy, then $G$ is Hamiltonian.

\end{theorem}
\begin{theorem}[Bedrossian, Chen and Schelp \cite{Bedrossian_Chen_Schelp}]\label{th7}
Let $G$ be a 2-connected graph. If $G$ is $\{K_{1,3},Z_1\}$-$f$-heavy, then $G$ is Hamiltonian.
\end{theorem}

\begin{theorem}[Li, Wei and Gao \cite{Li_Wei_Gao}]\label{th8}
Let $G$ be a 2-connected graph. If $G$ is $\{K_{1,3},B\}$-$f$-heavy, then $G$ is Hamiltonian.
\end{theorem}

\begin{theorem}[Chen,Wei and Zhang \cite{Chen_Wei_Zhang_0}]\label{th9}
Let $G$ be a 2-connected graph. If $G$ is $\{K_{1,3},N\}$-$f$-heavy, then $G$ is Hamiltonian.
\end{theorem}

In this paper, we prove the following two results.

\begin{theorem}\label{th10}
Let $G$ be a 2-connected graph. If $G$ is $\{K_{1,3},Z_2\}$-$f$-heavy, then $G$ is Hamiltonian.
\end{theorem}

\begin{theorem}\label{th11}
Let $G$ be a 2-connected graph. If $G$ is $\{K_{1,3},W\}$-$f$-heavy, then $G$ is Hamiltonian.
\end{theorem}

On the other hand, we have the following

\begin{remark}\label{re1}
Since $C_3$ is a clique, it contains no pairs of vertices at distance 2. For this reason, we say that every graph is $C_3$-\emph{f}-heavy. On the other hand, there indeed exist 2-connected claw-free graphs which are not Hamiltonian (a 3-connected claw-free non-Hamiltonian graph is shown in \cite{Matthews_Sumner}). Thus, not every 2-connected $\{K_{1,3},C_3\}$-\emph{f}-heavy graph is Hamiltonian.
\end{remark}

Note that every $P_i$-$f$-heavy ($i=4,5$) graph is $P_6$-$f$-heavy, every $Z_1$-$f$-heavy graph is $B$-$f$-heavy ($N$-$f$-heavy) and every $B$-$f$-heavy graph is $N$-$f$-heavy. Together with Remark \ref{re1} and Theorems \ref{th6}, \ref{th9}, \ref{th10} and \ref{th11}, we have

\begin{theorem}\label{th12}
Let $R$ and $S$ be connected graphs other than $P_3$ and let $G$ be a 2-connected graph. Then $G$ being $\{R,S\}$-f-heavy implies $G$ is Hamiltonian if and only if (up to symmetry) $R=K_{1,3}$ and $S=P_4,P_5,P_6,Z_1,Z_2,B,N$ or $W$.
\end{theorem}

Theorem \ref{th12} gives a complete answer to Problem \ref{pr1}.

Moreover, we can pose the following two problems naturally.

\begin{problem}\label{pr2}
Which two connected graphs $R$ and $S$ other than $P_3$ imply that every 2-connected $R$-\emph{o}-heavy and $S$-\emph{f}-heavy graph is Hamiltonian?
\end{problem}

\begin{problem}\label{pr3}
Which two connected graphs $R$ and $S$ other than $P_3$ imply that every 2-connected $R$-free and $S$-\emph{f}-heavy graph is Hamiltonian?
\end{problem}

By Theorem \ref{th1}, Problem \ref{pr2} is equivalent to the following two problems.

\setcounter{problem}{2}
\begin{subproblem}\label{subpr2.1}
Which connected graphs $S$ other than $P_3$ imply that every 2-connected claw-\emph{o}-heavy and $S$-\emph{f}-heavy graph is Hamiltonian?
\end{subproblem}

\begin{subproblem}\label{subpr2.2}
Which connected graphs $S$ other than $P_3$ imply that every 2-connected claw-$f$-heavy and $S$-$o$-heavy graph is Hamiltonian?
\end{subproblem}

For Problem \ref{subpr2.1}, by Theorem \ref{th1} and Remark \ref{re1}, we know that $S$ must be one of the graphs $P_4,P_5,P_6,Z_1,Z_2,B,N$ or $W$.

In this paper, instead of Theorems \ref{th10} and \ref{th11}, we prove the following stronger result.

\begin{theorem}\label{th13}
Let $G$ be a 2-connected graph. If $G$ is claw-$o$-heavy and $S$-f-heavy, where $S\in\{P_6,Z_2,W,N\}$, then $G$ is Hamiltonian.
\end{theorem}

As a corollary of Theorems \ref{th1} and \ref{th13}, we can get the following theorem, which gives a full answer to Problem \ref{subpr2.1}.

\begin{theorem}\label{th14}
Let $G$ be a 2-connected graph and $S$ be a connected graph other than $P_3$. If $G$ is claw-$o$-heavy, then $G$ being $S$-$f$-heavy implies $G$ is Hamiltonian if and only if $S=P_4,P_5,P_6,Z_1,Z_2,B,N$ or $W$.
\end{theorem}

For Problem \ref{subpr2.2}, we firstly notice that every claw-\emph{f}-heavy graph is also claw-\emph{o}-heavy. Secondly, it is known in \cite{Li_Ryjacek_Wang_Zhang} that there exists a 2-connected claw-free and $P_6$-\emph{o}-heavy graph which is not Hamiltonian. Thus the following result, which can be deduced from Theorem \ref{th4}, is an answer to Problem \ref{subpr2.2}.

\begin{corollary}\label{co1}
Let $G$ be a 2-connected graph and $S$ be a connected graph other than $P_3$. If $G$ is claw-$f$-heavy, then $G$ being $S$-$o$-heavy implies $G$ is Hamiltonian if and only if $S=P_4,P_5,C_3,Z_1,Z_2,B,N$ or $W$.
\end{corollary}

Similar to Problem \ref{pr2}, by Theorem \ref{th1}, Problem \ref{pr3} is equivalent to the following two problems.

\setcounter{problem}{3}
\setcounter{subproblem}{0}
\begin{subproblem}\label{subpr3.1}
Which connected graphs $S$ other than $P_3$ imply that every 2-connected claw-free and $S$-\emph{f}-heavy graph is Hamiltonian?
\end{subproblem}

\begin{subproblem}\label{subpr3.2}
Which connected graphs $S$ other than $P_3$ imply that every 2-connected claw-\emph{f}-heavy and $S$-free graph is Hamiltonian?
\end{subproblem}

For a given connected graph $H$, we notice that every $H$-free graph is also $H$-\emph{f}-heavy. Hence by Theorems \ref{th4}, \ref{th12} and Remark \ref{re1}, we have

\begin{corollary}\label{co2}
Let $G$ be a 2-connected graph and $S$ be a connected graph other than $P_3$. If $G$ is claw-free, then $G$ being $S$-$f$-heavy implies $G$ is Hamiltonian if and only if $S=P_4,P_5,P_6,Z_1,Z_2,B,N$ or $W$.
\end{corollary}

\begin{corollary}\label{co3}
Let $G$ be a 2-connected graph and $S$ be a connected graph other than $P_3$. If $G$ is claw-$f$-heavy, then $G$ being $S$-free implies $G$ is Hamiltonian if and only if $S=P_4,P_5,P_6,C_3,Z_1,Z_2,B,N$ or $W$.
\end{corollary}

Corollaries \ref{co2} and \ref{co3} answer Problems \ref{subpr3.1} and \ref{subpr3.2}, respectively.

Obviously, Theorem \ref{th13} extends Theorems \ref{th6}-\ref{th11}, and Theorem \ref{th14} extends Theorem \ref{th4}. Moreover, each of Theorems \ref{th12}, \ref{th14} and Corollaries \ref{co1}-\ref{co3} extends Theorem \ref{th1}.

In the next Section, we give some preliminaries. The proof of Theorem \ref{th13} is postponed to Section 3. In the last section, some remarks and one open problem are given.

\section{Preliminaries}

We begin this section with some additional terminology and notation.

Let $G$ be a graph, $H$ a subgraph of $G$ and $X$ a subset of $V(G)$. We use $G[X]$ to denote the subgraph of $G$ induced by $X$, and $G-H$ denotes the subgraph induced by $V(G)\setminus V(H)$. If $G'$ is a graph, then by $G[X]\cong G'$, we mean that $G[X]$ is isomorphic to the graph $G'$.

Throughout this paper, $k$ and $l$ will denote positive integers, and $s,t$ denote the integers which may be non-positive. For $s\leq t$, $[s,t]$ denotes the integer set $\{s,s+1,\ldots,t-1,t\}$ and $[u_s,u_t]$ denotes the set $\{u_s,u_{s+1},\ldots,u_{t-1},u_t\}$. If $[u_s,u_t]$ is a subset of the vertex set of a graph $G$, we use $G[u_s,u_t]$, instead of $G[[u_s,u_t]]$, to denote the subgraph induced by $[u_s,u_t]$ in $G$.

Let $P$ be a path and $u,v\in V(P)$. We use $P[u,v]$ to denote the subpath of $P$ from $u$ to $v$. Let $C$ be a cycle. We denote by $\overrightarrow{C}$ the cycle $C$ with a given orientation, and by $\overleftarrow{C}$ the same subgraph with the reverse orientation. For two vertices $u,v\in V(C)$, $\overrightarrow{C}[u,v]$ denotes the path from $u$ to $v$ on $\overrightarrow{C}$, and $\overleftarrow{C}[v,u]$ is the same path with the reverse direction. For a vertex $x\in \overrightarrow{C}$, we use $x^{+}$ to denote the successor of $x$ on $\overrightarrow{C}$, and $x^{-}$ denotes its predecessor. If $A\subseteq V(C)$, then set $A^{+}=\{x^{+}:x\in A\}$ and $A^{-}=\{x^{-}:x\in A\}$.

Let $G$ be a graph on $n$ vertices and $v$ be a vertex of $V(G)$. The vertex $v$ is called \emph{heavy} if its degree is at least $n/2$; otherwise we call it a \emph{light} vertex. A pair of nonadjacent vertices with degree sum at least $n$ is called a \emph{heavy pair} and a triangle such that every vertex in it is heavy is called a \emph{heavy triangle}. A cycle $C$ of $G$ is called \emph{heavy} if it contains all heavy vertices of $G$; it is called \emph{nonextendable} if there is not a longer cycle in $G$ which contains all the vertices of $C$.

In this paper, we need some concepts firstly introduced by Li et al. in \cite{Li_Ryjacek_Wang_Zhang}. To ensure the integrity of our paper, we rewrite them here.

Let $G$ be a graph and $C=x_1x_2,\ldots,x_t$ be a sequence of vertices in $V(G)$, where $t\geq 3$ be an integer. We denote $\widetilde{E}(G)=\{xy:xy\in E(G)$ or $d(x)+d(y)\geq n, x,y\in V(G)\}$, and say that $C$ is an \emph{Ore-cycle}, or in short, \emph{o-cycle}, if the vertices in $V(C)$ satisfy $x_ix_{i+1}\in \widetilde{E}(G)$, $i\in [1,t]$, where $x_1=x_{t+1}$.

Let $G$ be a graph and $\{x_1,x_2\}$, $\{y_1,y_2\}$ be two pairs of vertices in $V(G)$ with $x_1\neq x_2$ and $y_1\neq y_2$. We say that $D$ is an \emph{$(x_1x_2,y_1y_2)$-pair} if $D$ consists of two vertex-disjoint paths $P_1$ and $P_2$ such that
\begin{mathitem}
        \item the origin of $P_i$ is in $\{x_1,x_2\}$, and
        \item the terminus of $P_i$ is in $\{y_1,y_2\}$
\end{mathitem}
for $i=1,2$.

Let $G$ be a graph on $n\geq 2$ vertices and $x,y\in V(G)$ be two distinct vertices. Let $G'$ be a graph obtained by adding a (new) vertex $z$ to $G$ with two edges $zx$ and $zw$, where $w\neq x$ is an arbitrary vertex of $G$. Let $G''$ be a graph obtained by adding two (new) distinct vertices $x'$ and $y'$ to $G$ and three edges $xx',yy'$ and $x'y'$. We call $G'$ a \emph{1-extension of $G$ from $x$ to $z$}, and $G''$ a \emph{2-extension of $G$ from $(x,y)$ to $(x',y')$}.

Let $G$ be a graph and $x,y,z$ be three distinct vertices of $V(G)$. $G$ is called \emph{$(x,y,z)$-composed} if there exists a sequence of vertices $v_{-k},\ldots,v_0,\ldots,v_l$ ($k,l\geq 1$) and a sequence of graphs $D_1,D_2,\ldots,D_r$ ($r\geq 1$) such that
\begin{mathitem}
    \item[(1)]  $x=v_{-k}$, $y=v_0$ and $z=v_l$,
    \item[(2)]  $D_1$ is a triangle such that $V(D_1)=\{v_{-1},v_0,v_1\}$,
    \item[(3)]  $V(D_i)=[v_{-x_{i}},v_{y_i}]$ for some $x_i,y_i$, where $1\leq x_i\leq k$ and $1\leq y_i\leq l$, and $D_{i+1}$ satisfies one of the following conditions for $i\in [1,r-1]$:
        \begin{mathitem}
        \item $D_{i+1}$ is a 1-extension of $D_i$ from $v_{-x_{i}}$ to $v_{-x_{i}-1}$ or from $v_{y_{i}}$ to $v_{y_{i}+1}$,
        \item $D_{i+1}$ is a 2-extension of $D_i$ from $(v_{-x_{i}},v_{y_{i}})$ to $(v_{-x_{i}-1},v_{y_{i}+1})$,
        \end{mathitem}
    \item[(4)]  $D=D_r$ satisfies $V(D)=V(G)$.
\end{mathitem}
Without loss of generality, we call the sequence of vertices $v_{-k},\ldots,v_0,\ldots,v_l$ a \emph{canonical ordering} and the sequence of graphs $D_1,D_2,\ldots,D_r$ a \emph{canonical sequence of $D$}. We call the graph $D$ the \emph{carrier of $G$}.

Let $G$ be a graph, $C$ a cycle of $G$ and $x_1,x,x_2$ three distinct vertices on $C$. Let $P$ be the $(x_1,x_2)$-path on $C$ such that $x\in V(P)\setminus\{x_1,x_2\}$. The pair of vertices $(x_1,x_2)$ is said to be \emph{$x$-good on $C$}, if for some integer $i\in \{1,2\}$, there exits a vertex $x'\in V(P)\setminus \{x_i\}$ such that
\begin{mathitem}
\item there is an ($x,x_{3-i}$)-path $P'$ such that $V(P')=V(P)\setminus \{x_i\}$,
\item there is an ($xx_{3-i},x'x_i$)-pair $D$ such that $V(D)=V(P)$,
\item the degree sum of $x_i$ and $x'$ is at least $n$.
\end{mathitem}

Next, we list several known results needed in our proof.

\begin{lemma}[Li, Ryj\'a\v{c}ek, Wang and Zhang \cite{Li_Ryjacek_Wang_Zhang}]\label{le1}
Let $G$ be a graph and $C'$ be an $o$-cycle of $G$. Then there is a cycle $C$ of $G$ such that $V(C')\subseteq V(C)$.
\end{lemma}

\begin{lemma}[Li and Zhang \cite{Li_Zhang}]\label{le2}
Let $G$ be a 2-connected $K_{1,4}$-$o$-heavy graph and $C$ be a longest cycle of $G$. Then $C$ is a heavy cycle of $G$.
\end{lemma}

\begin{lemma}[Chv\'{a}tal and Erd\"{o}s \cite{Chvatal_Erdos}, Bondy \cite{Bondy}]\label{le3}
Let $G$ be a graph on $n$ vertices, $\overrightarrow{C}$ a nonextendable cycle in $G$, $H$ a component of $G-V(C)$, and $A$ the set of neighbours of $H$ on $C$. Then
\begin{mathitem}
    \item $A\cap A^{-}=\emptyset$, $A\cap A^{+}=\emptyset$, and $A^{-}$ and $A^{+}$ are independent sets,
    \item Each pair of vertices from $A^{-}$ or $A^{+}$ has degree sum smaller than $n$.
\end{mathitem}
\end{lemma}

\begin{lemma}[Li, Ryj\'a\v{c}ek, Wang and Zhang \cite{Li_Ryjacek_Wang_Zhang}]\label{le4}
Let $G$ be a composed graph and let $D$ and $v_{-k},\ldots,v_0,\ldots,v_l$ be a carrier and a canonical ordering of $G$. Then
\begin{mathitem}
    \item $D$ has a Hamilton $(v_0,v_{-k})$-path,
    \item for every $v_s\in V(G)\setminus\{v_{-k}\}$, $D$ has a spanning
      $(v_0v_l,v_sv_{-k})$-pair.
\end{mathitem}
\end{lemma}

\begin{lemma}[Li, Ryj\'a\v{c}ek, Wang and Zhang \cite{Li_Ryjacek_Wang_Zhang}]\label{le5}
Let $G$ be a graph, and $C$ be a cycle of $G$ with a given orientation. Let $P$ be an $(x,y)$-path of $G$ which is internally disjoint from $C$, where $x,y\in V(C)$. If there are vertices $x_1,x_2,y_1,y_2\in V(C)\setminus\{x,y\}$ such that
\begin{mathitem}
\item $x_2,x,x_1,y_1,y,y_2$ appear in the order along $\overrightarrow{C}$ (maybe $x_1=y_1$ or $x_2=y_2$);
\item $(x_1,x_2)$ is $x$-good on $C$; and
\item $(y_1,y_2)$ is $y$-good on $C$,
\end{mathitem}
then there is a cycle in $G$ which contains all the vertices in  $V(C)\cup V(P)$.
\end{lemma}

\section{Proof of Theorem 13}

Suppose that $G$ is a non-Hamiltonian graph on $n$ vertices.  Let $C$ be a longest cycle of $G$ and $c$ be the length of $C$. Then $c<n$ and $G-C\neq \emptyset$. Since $G$ is 2-connected, there is a path of length at least 2, internally-disjoint with $C$, that connects two vertices of $C$. Let $P=w_0w_1\ldots w_rw_{r+1}$ be such a path with $r$ as small as possible, where $w_0=u_0\in V(C)$ and $w_{r+1}=v_0\in V(C)$. Assume that the length of $\overrightarrow{C}[u_0,v_0]$ is $r_1+1$ and the length of $\overrightarrow{C}[v_0,u_0]$ is $r_2+1$. Obviously, $r_1+r_2+2=c$. We denote the cycle $C$ with a given orientation by $\overrightarrow{C}=u_0u_1\ldots u_{r_1}v_0u_{-r_2}\ldots u_{-1}u_0$ or by $\overrightarrow{C}=v_0v_1\ldots v_{r_2}u_0v_{-r_1}\ldots v_{-1}v_0$, where $u_l=v_{-r_1-1+l}$ and $u_{-k}=v_{r_2+1-k}$.

\setcounter{claim}{0}
\begin{claim}\label{cl1}
Let $x\in [w_1,w_r]$ and $y\in \{u_{-1},u_1,v_{-1},v_1\}$. Then $xy\notin \widetilde{E}(G)$.
\end{claim}

\begin{proof}
Without loss of generality, assume that $y=u_{-1}$. Suppose that $xy\in \widetilde{E}(G)$. Then $C'=P[u_0,x]xy\overleftarrow{C}[y,u_0]$ is an \emph{o}-cycle containing all the vertices of $C$ and longer than $C$. By Lemma \ref{le1}, there is a cycle longer than $C$, contradicting the choice of $C$.
\end{proof}

\begin{claim}\label{cl2}
$u_{-1}u_1\in \widetilde{E}(G)$, $v_{-1}v_1\in \widetilde{E}(G)$.
\end{claim}

\begin{proof}
Assume that $u_{-1}u_1\notin E(G)$. By Claim \ref{cl1}, we have $w_1u_{-1}\notin E(G)$ and $w_1u_1\notin E(G)$. Hence $G[\{u_{-1},u_0,u_1,w_1\}]\cong K_{1,3}$. Note that $w_1u_{-1}\notin \widetilde{E}(G)$ and $w_1u_1\notin \widetilde{E}(G)$ by Claim \ref{cl1}. Since $G$ is claw-\emph{o}-heavy, it follows that $d(u_{-1})+d(u_1)\geq n$. Thus, we obtain $u_{-1}u_1\in \widetilde{E}(G)$.

Similarly, we can prove $v_{-1}v_1\in \widetilde{E}(G)$.
\end{proof}

\begin{claim}\label{cl3}
$u_0v_{\pm 1}\notin \widetilde{E}(G)$, $v_0u_{\pm 1}\notin \widetilde{E}(G)$.
\end{claim}

\begin{proof}
Assume that $u_0v_1\in \widetilde{E}(G)$ or $u_0v_{-1}\in \widetilde{E}(G)$. Let $C'=P\overleftarrow{C}[v_0,u_1]u_1u_{-1}\overleftarrow{C}[u_{-1},v_1]v_1$ $u_0$ (if $u_0v_1\in \widetilde{E}(G)$) or $C'=P\overrightarrow{C}[v_0,u_{-1}]u_{-1}u_1\overrightarrow{C}[u_1,v_{-1}]v_{-1}u_0$ (if $u_0v_1\notin \widetilde{E}(G)$). By Claim \ref{cl2}, $C'$ is an \emph{o}-cycle containing all the vertices in $C$ and longer than $C$, a contradiction by Lemma \ref{le1}.

Similarly, we can prove $v_0u_{\pm 1}\notin \widetilde{E}(G)$.
\end{proof}

Let $u_{j_1}$ be the first vertex on $\overrightarrow{C}[u_1,u_{r_1}]$ such that $u_0u_{j_1}\notin E(G)$, $v_{j_2}$ be the first vertex on $\overrightarrow{C}[v_1, v_{r_2}]$ such that $v_0v_{j_2}\notin E(G)$. Obviously, we have $u_0u_1\in E(G)$ and $v_0v_1\in E(G)$. By Claim \ref{cl3}, we know that $u_0u_{r_1}\notin E(G)$ and $v_0v_{r_2}\notin E(G)$. Thus, $u_{j_1}$, $v_{j_2}$ exist, where $2\leq j_1\leq r_1$ and $2\leq j_2\leq r_2$.

\begin{claim}\label{cl4}
Let $w\in [w_1,w_r], u\in [u_1,u_{j_1}]$ and $v\in [v_1,v_{j_2}]$. Then we have
\begin{mathitem}
\item $wu\notin \widetilde{E}(G), wv\notin \widetilde{E}(G)$,
\item $v_0u\notin \widetilde{E}(G), u_0v\notin \widetilde{E}(G)$,
\item $uv\notin \widetilde{E}(G)$.
\end{mathitem}
\end{claim}

\begin{proof}
($i$) Assume that $wu\in \widetilde{E}(G)$. If $u=u_1$, then we get a contradiction by Claim \ref{cl1}. If $u=u_2$, then let $C'=P[u_0,w]wu\overrightarrow{C}[u,u_{-1}]u_{-1}u_1u_0$. If $u\in [u_3,u_{j_1}]$, then let $C'=P[u_0,w]wu\overrightarrow{C}[u,u_{-1}]u_{-1}u_1\overrightarrow{C}[u_1,u^{-}]u^{-}u_0$. By Claim \ref{cl2}, $C'$ is an \emph{o}-cycle longer than $C$ and contains all the vertices in $C$. Therefore, there is a cycle longer than $C$ by Lemma \ref{le1}, a contradiction.

The second assertion can be proved similarly.

($ii$) Assume that $v_0u\in \widetilde{E}(G)$. By Claim \ref{cl3}, we have $v_0u_1\notin \widetilde{E}(G)$. Hence we have $u\in [u_2,u_{j_1}]$. Then $C'=v_0u\overrightarrow{C}[u,v_{-1}]v_{-1}v_1\overrightarrow{C}[v_1,u_{-1}]u_{-1}u_1\overrightarrow{C}[u_1,u^{-}]u^{-}u_0P$ is an \emph{o}-cycle longer than $C$ and contains all the vertices in $C$ by Claim \ref{cl2}. By Lemma \ref{le1}, there is a cycle longer than $C$, contradicting the choice of $C$.

Similarly, we can prove that $u_0v\notin \widetilde{E}(G)$.

($iii$) Assume that $uv\in \widetilde{E}(G)$. By Claim \ref{cl2}, we have $u_{-1}u_1\in \widetilde{E}(G)$ and $v_{-1}v_1\in \widetilde{E}(G)$. Then
$C'=Pv_0v^{-}\overleftarrow{C}[v^{-},v_1]v_1v_{-1}\overleftarrow{C}[v_{-1},u]uv\overrightarrow{C}[u,u_{-1}]u_{-1}u_1\overrightarrow{C}[u_1,u^{-}]u^{-}u_0$ (if $u\neq u_1$ and $v\neq v_1$) or $C'=Pv_0v^{-}\overleftarrow{C}[v^{-},v_1]v_1v_{-1}\overleftarrow{C}[v_{-1},u_1]u_1v\overrightarrow{C}[v,u_0]$ (if $u=u_1$ and $v\neq v_1$) or $C'=P\overleftarrow{C}[v_0,u]uv_1\overrightarrow{C}[v_1,u_{-1}]u_{-1}u_1\overrightarrow{C}[u_1,u^{-}]u^{-}u_0$ (if $u\neq u_1$ and $v=v_1$) or $C'=P\overleftarrow{C}[v_0,u_1]u_1v_1\overrightarrow{C}[v_1,u_0]$ (if $u=u_1$ and $v= v_1$) is an \emph{o}-cycle longer than $C$ and contains all the vertices in $C$. By Lemma \ref{le1}, there is a cycle containing all the vertices in $V(P)\cup V(C)$, a contradiction.
\end{proof}

\begin{claim}\label{cl5}
$d(u_0)+d(v_0)<n$.
\end{claim}
\begin{proof}
Let $P'=u_0x_1x_2,\ldots,x_{r'}v_0$ be a $(u_0,v_0)$-path internally-disjoint with $C$ such that its length is as large as possible.

\begin{subclaim}\label{su5.1}
$d_{G-C}(u_0)+d_{G-C}(v_0)\leq 2r'$.
\end{subclaim}

\begin{proof}
We will show that all the neighbors of $u_0$ in $G-C$ are contained in $V(P')$. Assume not. Let $x'_1$ be a neighbor of $u_0$, which is in $V(G-C)$ but not in $V(P')$. Obviously, we have $x_1u_1,x'_1u_1\notin \widetilde{E}(G)$; otherwise there is a cycle longer than $C$ by Lemma \ref{le1}, contradicting the choice of $C$. If $x_1x'_1\notin E(G)$, then $G[\{u_0,u_1,x_1,x'_1\}]\cong K_{1,3}$. Note that $G$ is claw-\emph{o}-heavy. Thus, we have $d(x_1)+d(x'_1)\geq n$. This implies that either $x_1$ or $x'_1$ is heavy. However, it follows from the fact $G$ is claw-\emph{o}-heavy and Lemma \ref{le2} that $C$ is heavy, a contradiction. If $x_1x'_1\in E(G)$, then $P''=u_0x_1'x_1P'[x_1,v_0]$ is a $(u_0,v_0)$-path internally-disjoint with $C$ and longer than $P'$, contradicting the choice of $P'$. Therefore, it follows that $d_{G-C}(u_0)\leq r'$.

Similarly, we can obtain $d_{G-C}(v_0)\leq r'$ and the proof of this claim is complete.
\end{proof}

Let $u_k$ be the last vertex on $\overrightarrow{C}[u_1,u_{r_1}]$ such that $u_0u_k\in E(G)$, $u_l$ be the first vertex on $\overrightarrow{C}[u_{k+1},u_{r_1}]$ such that $v_0u_l\in E(G)$.

\begin{subclaim}\label{su5.2}
For every vertex $u_{k'}\in N_{C[u_1,u_{k-1}]}(u_0)\cup \{u_0\}$, $v_0u_{k'+1}\notin E(G)$.
\end{subclaim}

\begin{proof}
By Claim \ref{cl3}, we have $v_0u_1\notin E(G)$. If $u_{k'}\neq u_0$, assume that $u_{k'}u_0\in E(G)$ and $v_0u_{k'+1}\in E(G)$. Then
$C'=Pv_0u_{k'+1}\overrightarrow{C}[u_{k'+1},v_{-1}]v_{-1}v_1\overrightarrow{C}[v_1,u_{-1}]u_{-1}u_1\overrightarrow{C}[u_1,u_{k'}]u_{k'}u_0$ is an \emph{o}-cycle containing all the vertices in $V(P)\cup V(C)$. Thus, there is a cycle longer than $C$ by Lemma \ref{le1}, a contradiction.
\end{proof}

\begin{subclaim}\label{su5.3}
$|[u_{k+1},u_{l-1}]|=l-k-1\geq r'$.
\end{subclaim}
\begin{proof}
Assume that $|[u_{k+1},u_{l-1}]|< r'$. Then $C'=P'v_0u_l\overrightarrow{C}[u_l,v_{-1}]v_{-1}v_1\overrightarrow{C}[v_1,u_{-1}]u_{-1}u_1$\\$\overrightarrow{C}[u_1,u_k]u_ku_0$ is an \emph{o}-cycle which contains all the vertices in $V(C)\setminus [u_{k+1},u_{l-1}]\cup V(P')$ and $|V(C')|> c$. Hence there is a cycle longer than $C$ by Lemma \ref{le1}, a contradiction.
\end{proof}

The following claim is obvious.

\begin{subclaim}\label{su5.4}
$N_C(u_0)\cap [u_{k+1},u_{r_1}]=\emptyset$.
\end{subclaim}

Let $d_1=|N_{C[u_1,u_k]}(u_0)|$. Then by Claim \ref{su5.4}, $d_{C[u_1,u_{r_1}]}(u_0)=d_1$. By Claims \ref{su5.2} and \ref{su5.3}, we have $d_{C[u_1,u_r]}(v_0)=d_{C[u_1,u_k]}(v_0)+d_{C[u_{k+1},u_{l-1}]}(v_0)+d_{C[u_l,u_{r_1}]}(v_0)\leq k-d_1+r_1-l+1\leq r_1-d_1-r'$. Thus $d_{C[u_1,u_{r_1}]}(u_0)+d_{C[u_1,u_{r_1}]}(v_0)\leq r_1-r'$, and similarly, $d_{C[v_1,v_{r_2}]}(u_0)+d_{C[v_1,v_{r_2}]}(v_0)\leq r_2-r'$. Hence $d_{C}(u_0)+d_{C}(v_0)\leq r_1+r_2-2r'+2=c-2r'$. Note that $d_{G-C}(u_0)+d_{G-C}(v_0)\leq 2r'$ by Claim \ref{su5.1}. Therefore, $d(u_0)+d(v_0)\leq c<n$.
\end{proof}

\begin{claim}\label{cl6}
Either $u_{-1}u_1\in E(G)$ or $v_{-1}v_1\in E(G)$.
\end{claim}

\begin{proof}
Assume that $u_{-1}u_1\notin E(G)$ and $v_{-1}v_1\notin E(G)$. By Claim \ref{cl2}, we have $d(u_{-1})+d(u_1)\geq n$ and $d(v_{-1})+d(v_1)\geq n$. Thus, we obtain $d(u_{-1})+d(v_{-1})\geq n$ or $d(u_1)+d(v_1)\geq n$, contradicting Lemma \ref{le3}.
\end{proof}

We divide the remainder part of the proof into three cases.

\begin{case}
$S=P_6$.
\end{case}

If $u_0v_0\notin E(G)$, then by Claim \ref{cl4} and the fact that $P$ is a required shortest path, $R_1=G[\{u_0,u_{j_1-1},u_{j_1},v_0,v_{j_2-1},v_{j_2},w_1,w_2,\ldots,w_r\}]\cong P_{r+6}$. Since $G$ is $P_6$-\emph{f}-heavy, $G$ is $P_{6+r}$-\emph{f}-heavy. Note that $C$ is heavy by Lemma \ref{le2}. Hence each of $\{w_1,w_r\}$ is light. It follows from the fact $d_{R_1}(w_1,u_{j_1-1})=d_{R_1}(w_r,v_{j_2-1})=2$ that each of $\{u_{j_1-1},v_{j_2-1}\}$ is heavy. Then we have $u_{j_1-1}v_{j_2-1}\in \widetilde{E}(G)$, and it contradicts to Claim \ref{cl4} ($iii$).

If $u_0v_0\in E(G)$, then by Claim \ref{cl4}, $R_2=G[\{u_0,u_{j_1-1},u_{j_1},v_0,v_{j_2-1},v_{j_2}\}]\cong P_{6}$. By Claim \ref{cl5}, either $u_0$ or $v_0$ is light. Without loss of generality, assume that $u_0$ is light. Since $G$ is $P_6$-\emph{f}-heavy and $d_{R_2}(u_0,u_{j_1})=d_{R_2}(u_0,v_{j_2-1})=2$, we have each of $\{u_{j_1},v_{j_2-1}\}$ is heavy. It follows that $u_{j_1}v_{j_2-1}\in \widetilde{E}(G)$, which contradicts to Claim \ref{cl4} ($iii$).

\begin{case}
$S=Z_2$
\end{case}

\begin{subcase}\label{suc1}
$u_0v_0\in E(G)$.
\end{subcase}
We claim that $r=1$. If $r\geq 2$, then by the choice of $P$, we have $w_1v_0\notin E(G)$ and $w_ru_0\notin E(G)$. By Claims \ref{cl1} and \ref{cl3}, we obtain $w_1u_1,w_rv_{-1},u_1v_0,v_{-1}u_0\notin \widetilde{E}(G)$. Hence $G[\{w_1,u_0,u_1,v_0\}]\cong K_{1,3}$ and $G[\{u_0,v_0,v_{-1},w_r\}]\cong K_{1,3}$. Note that $G$ is claw-\emph{o}-heavy. It follows that $d(w_1)+d(v_0)\geq n$ and $d(w_r)+d(u_0)\geq n$. This means that either $d(w_1)+d(w_r)\geq n$ or $d(u_0)+d(v_0)\geq n$. However, we have $d(w_1)+d(w_r)<n$ since $C$ is heavy. By Claim \ref{cl5}, we can obtain $d(u_0)+d(v_0)<n$, a contradiction.

By Claim \ref{cl4}, $R_1=G[\{u_0,w_1,v_0,u_{j_1-1},u_{j_1}\}]\cong Z_2$ and $R_2=G[\{u_0,w_1,v_0,v_{j_2-1},v_{j_2}\}]\cong Z_2$. Note that $d_{R_1}(w_1,u_{j_1-1})=d_{R_2}(w_1,v_{j_2-1})=2$ and $w_1$ is light. It follows from the condition $G$ is $Z_2$-\emph{f}-heavy that each of $\{u_{j_1-1},v_{j_2-1}\}$ is heavy. Thus, $u_{j_1-1}v_{j_2-1}\in \widetilde{E}(G)$, contradicting Claim \ref{cl4} ($iii$).
\begin{subcase}
$u_0v_0\notin E(G)$.
\end{subcase}

By Claim \ref{cl6}, we have $u_{-1}u_1\in E(G)$ or $v_{-1}v_1\in E(G)$. Without loss of generality, suppose that $u_{-1}u_1\in E(G)$. By Claims \ref{cl1}, \ref{cl3} and the hypothesis that $u_0v_0\notin E(G)$, we have $R_1=G[\{w_1,w_2,u_{-1},u_0,u_1\}]\cong Z_2$. Note that $w_1$ is light and $d_{R_1}(w_1,u_1)=d_{R_1}(w_1,u_{-1})=2$. It follows from the fact $G$ is $Z_2$-\emph{f}-heavy that each of $\{u_{-1},u_1\}$ is heavy.

If $v_{-1}v_1\notin E(G)$, then $G[\{w_r,v_{-1},v_0,v_1\}]\cong K_{1,3}$ by Claim \ref{cl1}. Since $G$ is claw-\emph{o}-heavy, we have $d(v_{-1})+d(v_1)\geq n$ by Claim \ref{cl1}. Hence we can obtain $d(u_{-1})+d(u_1)\geq n$ or $d(v_{-1})+d(v_1)\geq n$, which contradicts to Lemma \ref{le3}.

If $v_{-1}v_1\in E(G)$, then $R_1=G[\{w_{r-1},w_r,v_{-1},v_0,v_1\}]\cong Z_2$.  Note that $w_r$ is light and $d_{R_1}(w_r,v_{-1})=d_{R_1}(w_r,v_1)=2$. It follows from the fact $G$ is $Z_2$-\emph{f}-heavy that $d(v_{-1})\geq n/2$ and $d(v_1)\geq n/2$. Hence $d(u_{-1})+d(v_{-1})\geq n$, contradicting Lemma \ref{cl3}.

\begin{case}
$S=W$ or $S=N$.
\end{case}
When $S=W$ and $u_0v_0\in E(G)$, similarly as Case 2.1, we can prove that $r=1$. By Claim \ref{cl4}, $R_1=G[\{u_0,w_1,v_0,u_{j_1-1},u_{j_1},v_{j_2-1}\}]\cong W$. Note that $d_{R_1}(w_1,u_{j_1-1})=d_{R_1}(w_1,v_{j_2-1})=2$ and $w_1$ is light. Since $G$ is $W$-\emph{f}-heavy, it follows that each of $\{u_{j_1-1},v_{j_2-1}\}$ is heavy. Thus, $u_{j_1-1}v_{j_2-1}\in \widetilde{E}(G)$, contradicting Claim \ref{cl4} ($iii$).

Now we can suppose that $S=W$ and $u_0v_0\notin E(G)$ or $S=N$.

By Claim \ref{cl6}, we have $u_{-1}u_1\in E(G)$ or $v_{-1}v_1\in E(G)$. Suppose, without loss of generality, that $u_{-1}u_1\in E(G)$. Note that $G[u_{-1},u_1]$ is $(u_{-1},u_0,u_1)$-composed.

\begin{claim}\label{cl7}
If $G[u_{-k},u_l]$ is $(u_{-k},u_0,u_l)$-composed with the canonical ordering $u_{-k},u_{-k+1},\ldots,$\\$u_{l-1},u_l$, then $k\leq r_2-2$ and $l\leq r_1-2$.
\end{claim}

\begin{proof}
Suppose that $k\geq r_2-1$. Let $D_1$, $D_2$,\ldots, $D_r$ be a canonical sequence of $G[u_{-k},u_l]$ corresponding to the canonical ordering $u_{-k},u_{-k+1},\ldots,u_{l-1},u_l$. Consider the graph $D'=D_{\widehat{-r_2+1}}$, where $\widehat{-r_2+1}$ be the smallest integer such that $u_{-r_2+1}\in D_{\widehat{-r_2+1}}$. Note that the index $\widehat{-r_2+1}$ exists since $0\geq -r_2+1\geq -k$. By Lemma \ref{le4}, there is a $(u_0,u_l)$-path $P'$ satisfying $V(P')=[u_{-r_2+1},u_l]$. Then $C'=v_1v_0P[v_0,u_0]P'[u_0,u_l]\overrightarrow{C}[u_l,v_{-1}]v_{-1}v_1$ is an \emph{o}-cycle such that $V(C)\cup V(P)\subseteq V(C')$, and there is a cycle longer than $C$ by Lemma \ref{le1}, a contradiction. Similarly, we can prove that $l\leq r_1-2$.
\end{proof}

\begin{claim}\label{cl8}
If $G[u_{-k},u_l]$ is $(u_{-k},u_0,u_l)$-composed with the canonical ordering $u_{-k},u_{-k+1},\ldots,$\\$u_{l-1},u_l$, where $k\leq r_2-2$ and $l\leq r_1-2$, and moreover the following two statements hold: \\
($i$) there is not a heavy pair in $G[u_{-k-1},u_{l+1}]$,\\
($ii$) there is not a heavy triangle in $G[u_{-k-1},u_{l+1}]$,\\
then one of the following is true:\\
(1) $G[u_{-k-1},u_l]$ is $(u_{-k-1},u_0,u_l)$-composed with the canonical ordering $u_{-k-1},u_k,\ldots,u_l$,\\
(2) $G[u_{-k},u_{l+1}]$ is $(u_{-k},u_0,u_{l+1})$-composed with the canonical ordering $u_{-k},u_k,\ldots,u_{l+1}$,\\
(3) $G[u_{-k-1},u_{l+1}]$ is $(u_{-k-1},u_0,u_{l+1})$-composed with the canonical ordering $u_{-k-1},u_k,\ldots,$\\$u_{l+1}$.
\end{claim}

\begin{proof}
Assume not. Then we have $u_{-k-1}u_s\notin E(G)$ for every vertex $s\in [-k+1,l]$, $u_su_{l+1}\notin E(G)$ for every vertex $s\in [-k,l-1]$, and $u_{-k-1}u_{l+1}\notin E(G)$.

\begin{subclaim}\label{su8.1}
For any vertex $u_s\in [u_{-k-1},u_{l+1}]\setminus \{u_0\}$ and $w\in \{w_1,w_2\}$, we have $u_sw\notin \widetilde{E}(G)$. Moreover, we have $u_0w_2\notin E(G)$ if $u_0v_0\notin E(G)$.
\end{subclaim}

\begin{proof}
Without loss of generality, suppose that $u_sw\in \widetilde{E}(G)$ and $s>0$. If $s=1$, then $u_1w\notin \widetilde{E}(G)$ by Claim \ref{cl1} or \ref{cl3}. Now assume that $s\in [2,{l+1}]$. Since $G[u_{-k},u_l]$ is $(u_{-k},u_0,u_l)$-composed, there exists an integer $t\in [-k,-1]$ such that $G[u_t,u_{s-1}]$ is $(u_t,u_0,u_{s-1})$-composed. By Lemma \ref{le4}, there is a $(u_0,u_t)$-path $P'$ such that $V(P')=[u_t,u_{s-1}]$. Let $C'=P'[u_0,u_t]\overleftarrow{C}[u_t,u_s]u_swP[w,u_0]$ (if $w\neq v_0$) or $C'=P'[u_0,u_t]\overleftarrow{C}[u_t,v_1]v_1\break v_{-1}\overleftarrow{C}[v_{-1},u_s]u_sv_0$$P[v_0,u_0]$ (if $w=v_0$). Clearly, $C'$ is an \emph{o}-cycle such that $V(C)\subseteq V(C')$ and $|V(C')|>|V(C)|$. By Lemma \ref{le1}, a contradiction.

Moreover if $r\geq 2$, then $u_0w_2\notin E(G)$ by the choice of $P$. If $r=1$ and $u_0v_0\notin E(G)$, then $u_0w_2=u_0v_0\notin E(G)$.
\end{proof}

Let $G'=G[[u_{-k-1},u_l]\cup \{w_1,w_2\}]$ and $G''=G[[u_{-k-1},u_{l+1}]\cup \{w_1,w_2\}]$.

\begin{subclaim}\label{su8.2}
If $S=W$ and $u_0v_0\notin E(G)$, then $G''$ and $G'$ are $\{K_{1,3},W\}$-free; If $S=N$, then $G''$ and $G'$ are $\{K_{1,3},N\}$-free.
\end{subclaim}

\begin{proof}
Note that $G'$ is an induced subgraph of $G''$. Hence we only need to prove that $G''$ satisfies the required property.

Assume that $G''$ contains an induced claw. Without loss of generality, let $H$ be the claw. If $V(H)\subseteq [u_{-k-1},u_{l+1}]$, then since $G$ is claw-\emph{o}-heavy, there is a heavy pair in $[u_{-k-1},u_{l+1}]$, which contradicts to condition ($i$) of Claim \ref{cl8}. If $w_1\in V(H)$ or $w_2\in V(H)$, then by Claim \ref{su8.1}, $d_H(w_1)\leq d_{G''}(w_1)=2$ and $d_H(w_2)\leq d_{G''}(w_2)\leq 2$. Hence $u_0$ is the center of $H$ and the other two end vertices $x_1,x_2$ of $H$ are in $[u_{-k-1},u_{l+1}]$. By Claim \ref{su8.1}, $w_1x_1,w_1x_2\notin \widetilde{E}(G)$. Since $G$ is claw-\emph{o}-heavy, $x_1,x_2$ is heavy pair in $G[u_{-k-1},u_{l+1}]$, a contradiction.

If $S=W$, then assume that $G''$ contains an induced subgraph $H\cong W$ depicted in Figure 1. Obviously, one vertex of $\{a_1,c_1\}$ and one vertex of $\{a_2,b_1\}$ and one vertex of $\{a_3,b_2\}$ are heavy. Hence there are at least three heavy vertices in $G''$. By Lemma \ref{le2} and the choice of $C$, $w_1$ is not heavy. Thus there is at least one heavy vertex in $[u_{-k-1},u_{l+1}]\backslash\{u_0\}$. By Claim 8.1, $w_2$ is not heavy. Thus there are at least three heavy vertices in $[u_{-k-1},u_{l+1}]$. If these three heavy vertices are adjacent to each other, then there is a heavy triangle in $[u_{-k-1},u_{l+1}]$, a contradiction. Otherwise, there are a heavy pair in $[u_{-k-1},u_{l+1}]$, a contradiction.

If $S=N$, then assume that $G''$ contains an induced subgraph $H\cong N$ depicted in Figure 1. Obviously, one vertex of $\{a_1,b_2\}$ and one vertex of $\{a_2,b_3\}$ and one vertex of $\{a_3,b_1\}$ are heavy. Hence there are at least three heavy vertices in $G''$. Similarly as the analysis above, we can deduce a contradiction.

The proof is complete.
\end{proof}

Now, we define $N_i=\{x\in V(G'): d_{G'}(x,u_{-k-1})=i\}$. Therefore, we have $N_0=\{u_{-k-1}\}$, and $N_1=\{u_{-k}\}$ by the fact that $u_{-k-1}u_s\notin E(G)$, where $s\in [-k+1,l]$. Without loss of generality, we assume $u_0\in N_j$, where $j\geq 2$. Then we have $w_1\in N_{j+1}$, $w_2\in N_{j+1}$ if $u_0w_2\in E(G)$, and $w_2\in N_{j+2}$ if $u_0w_2\notin E(G)$ by Claim \ref{su8.1}.

\begin{subclaim}\label{su8.3}
For $i\in [1,j]$, $N_i$ is a clique.
\end{subclaim}

\begin{proof}
Suppose that $|N_i|=1$ for some $i\in [2,j-1]$, and we set $N_i=\{x\}$. Then $x$ is a cut vertex of $G[u_{-k},u_l]$, contradicting the fact that $G[u_{-k},u_l]$ is 2-connected. Thus, we have $|N_i|\geq 2$ for every integer $i\in [2,j-1]$.

Now, we prove this claim by induction on $i$. If $i=1$, it is trivially true. If $i=2$, suppose that there exist $x,y\in N_2$ such that $xy\notin E(G)$, then $G[\{x,y,u_{-k},u_{-k-1}\}]\cong K_{1,3}$, a contradiction. Hence the claim is true when $i=2$. Now, we assume $3\leq i\leq j$, and we have each of $N_{i-3}$, $N_{i-2}$, $N_{i-1}$, $N_{i+1}$ is nonempty, and $|N_{i-1}|\geq 2$.

\noindent\textbf{Case A.}
$i<j$ or $i=j$ and $w_2u_0\notin E(G)$.

Note that $N_{i+2}$ is nonempty in this case. Let $x$ be a vertex of $N_i$ such that $y$ is a neighbor of it in $N_{i+1}$ which has a neighbor $z$ in $N_{i+2}$. For every vertex $x'\in N_i\backslash \{x\}$, we will show that $xx'\in E(G)$. Assume that $xx'\notin E(G)$. If $x'y\in E(G)$, then $G[\{x,x',y,z\}]\cong K_{1,3}$, a contradiction. If $x$ and $x'$ have a common neighbor in $N_{i-1}$, let it be $v$ and $w$ be a neighbor of $v$ in $N_{i-2}$. Then $G[\{x,x',v,w\}]\cong K_{1,3}$, a contradiction. Thus we assume $x$ and $x'$ have no common neighbors in $N_{i-1}$.

Let $v$ be a neighbor of $x$ in $N_{i-1}$ and $v'$ be a neighbor of $x'$ in $N_{i-1}$. By induction hypothesis, we have $vv'\in E(G)$. Let $w$ be a neighbor of $v$ in $N_{i-2}$ and $u$ be a neighbor of $w$ in $N_{i-3}$. If $wv'\notin E(G)$, then $G[\{x,v,v',w\}]\cong K_{1,3}$, a contradiction. Hence it follows that $v'w\in E(G)$. Now, we have $G[\{y,x,x',v,v',w\}]\cong W$ and  $G[\{x,x',v,v',w,u\}]\cong N$, a contradiction to Claim \ref{su8.2}. Therefore, for every vertex $x'\in N_i\setminus \{x\}$, $xx'\in E(G)$.

If there exist $x',x''\in N_i\setminus \{x\}$ such that $x'\neq x''$ and $x'x''\notin E(G)$, then we have $xx'\in E(G)$ and $xx''\in E(G)$. Note that $y$ is a neighbor of $x$ in $N_{i+1}$. If $x'y\in E(G)$ or $x''y\in E(G)$, then similar to the case of $x$ given above, we have $x'x''\in E(G)$, a contradiction. It follows that $x'y\notin E(G)$ and $x''y\notin E(G)$. Therefore, $G[\{x,x',x'',y\}]\cong K_{1,3}$, a contradiction.

\noindent\textbf{Case B.}
$i=j$ and $w_2u_0\in E(G)$.

We prove that for every vertex $x\in N_j\setminus\{u_0\}$, $xu_0\in E(G)$. Assume that $xu_0\notin E(G)$. If $x$ and $u_0$ have a common neighbor in $N_{j-1}$, let it be $v$ and $w$ be a neighbor of $v$ in $N_{j-2}$. Then $G[\{x,u_0,v,w\}]\cong K_{1,3}$, a contradiction. Thus we assume $x$ and $u_0$ have no common neighbors in $N_{j-1}$.

Let $v'$ be a neighbor of $u_0$ in $N_{j-1}$ and $v$ be a neighbor of $x$ in $N_{j-1}$. By induction hypothesis, we have $vv'\in E(G)$. Note that $u_0v\notin E(G)$ and $xv'\notin E(G)$. Let $w'$ be a neighbor of $v'$ in $N_{i-2}$ and $u'$ be a neighbor of $w'$ in $N_{j-3}$. If $w'v\notin E(G)$, then $G[\{u_0,v,v',w'\}]\cong K_{1,3}$, a contradiction. Hence it follows that $vw'\in E(G)$. Now we have $G[\{u_0,x,v,v',w',u'\}]\cong N$ and $G[\{u_0,x,v,v',w',w_1\}]\cong W$, contradicting Claim \ref{su8.2}. Therefore, for every vertex $x\in N_j\setminus \{u_0\}$, $xu_0\in E(G)$.

If there exist $x',x''\in N_j\setminus \{u_0\}$ such that $x'\neq x''$ and $x'x''\notin E(G)$. By the analysis above, we have $u_0x'\in E(G)$ and $u_0x''\in E(G)$. Note that $x',x''\neq w_2$. By Claim \ref{su8.1}, we have $w_1x'\notin E(G)$ and $w_1x''\notin E(G)$. Hence $G[\{w_1,u_0,x',x''\}]\cong K_{1,3}$, a contradiction.

The proof is complete.
\end{proof}

\begin{subclaim}\label{su8.4}
If $S=W$ and $u_0v_0\notin E(G)$, then $N_{G'}(u_0)\setminus \{w_1\}$ is a clique; If $S=N$, then $N_{G'}(u_0)\setminus \{w_1,w_2\}$ is a clique.
\end{subclaim}

\begin{proof}
Suppose not. If $S=W$ and $u_0v_0\notin E(G)$, then let $x,y\in N_{G'}(u_0)\setminus \{w_1\}$ are two vertices such that $xy\notin E(G)$. By Claim \ref{su8.1}, we have $x,y\neq w_2$, $w_1x\notin E(G)$ and $w_1y\notin E(G)$. Hence we have $G[\{x,y,u_0,w_1\}]\cong K_{1,3}$, a contradiction. If $S=N$, then suppose $x,y\in N_{G'}(u_0)\setminus \{w_1,w_2\}$ are two vertices such that $xy\notin E(G)$. Hence we have $G[\{x,y,u_0,w_1\}]\cong K_{1,3}$ by Claim \ref{su8.1}, a contradiction.
\end{proof}

\begin{subclaim}\label{su8.5}
$[u_{-k},u_l]\subseteq \bigcup_{i=1}^{j}N_i$.
\end{subclaim}

\begin{proof}
Assume there exists a vertex $x\in [u_{-k},u_l]$ such that $x\in N_{j+1}$. Let $y$ be a neighbor of $x$ in $N_j$, $z$ be a neighbor of $u_0$ in $N_{j-1}$ and $v$ be a neighbor of $z$ in $N_{j-2}$. Note that $x,z\notin\{w_1,w_2\}$. Then we have $xu_0\notin E(G)$, since otherwise $xz\in E(G)$ by Claim \ref{su8.4}, and this implies that $x\notin N_{j+1}$, a contradiction. By Claim \ref{su8.1}, we have $yw_1\notin E(G)$. Note that $yu_0\in E(G)$ by Claim \ref{su8.3}. If $yz\notin E(G)$, then $G[\{y,z,u_0,w_1\}]\cong K_{1,3}$, a contradiction. Now we assume $yz\in E(G)$. If $S=W$ and $u_0v_0\notin E(G)$, we have $u_0w_2\notin E(G)$ by Claim \ref{su8.1}. Hence $G[\{x,y,z,u_0,w_1,w_2\}]\cong W$, and it contradicts to Claim \ref{su8.2}. If $S=N$, then $G[\{x,y,z,u_0,w_1,v\}]\cong N$, which also contradicts to Claim \ref{su8.2}.

The proof is complete.
\end{proof}

It follows from Claim \ref{su8.5} that $u_l\in N_j$ or $u_l\in N_i$ where $i\in [2,j-1]$.

If $u_l\in N_j$, then let $x$ be a neighbor of $u_0$ in $N_{j-1}$ and $y$ be a neighbor of $x$ in $N_{j-2}$. Since $u_l,u_0\in N_j$, we have $u_lu_0\in E(G)$ by Claim \ref{su8.3}. By Claim \ref{su8.1}, we have $u_lw_1\notin E(G)$ and $xw_1\notin E(G)$. If $u_lx\notin E(G)$, then $G[\{w_1,u_0,u_l,x\}]\cong K_{1,3}$, a contradiction. Otherwise, we have $u_lx\in E(G)$. If $S=W$ and $u_0v_0\notin E(G)$, we have $u_0w_2\notin E(G)$ by Claim \ref{su8.1}. By Claim \ref{su8.1}, $G[\{x,y,u_0,u_l,w_1,w_2\}]\cong W$. If $S=N$, then by the fact that $u_su_{l+1}\notin E(G)$, where $s\in [-k,l-1]$ and Claim \ref{su8.1}, we have $G[\{x,y,u_0,u_l,u_{l+1},w_1\}]\cong N$. In each case, it contradicts to Claim \ref{su8.2}.

Now assume that $u_l\in N_i$, where $i\in [2,j-1]$ and $j\geq 3$. If $u_l$ has a neighbor in $N_{i+1}$, without loss of generality, let $x$ be a required vertex and $y$ be a neighbor of $u_l$ in $N_{i-1}$. Note that $i+1\leq j$, and this implies that $x\neq w_1,w_2$. By the fact that $u_su_{l+1}\notin E(G)$, where $s\in [-k,l-1]$, we have $G[\{u_l,u_{l+1},x,y\}]\cong K_{1,3}$, a contradiction. Then we assume $u_l$ has no neighbors in $N_{i+1}$.

Since $|N_i|\geq 2$, we can choose $x\in N_i$ be a vertex other than $u_l$ such that $y$ is a neighbor of $x$ in $N_{i+1}$ which has a neighbor $z$ in $N_{i+2}$. Let $u$ be a neighbor of $x$ in $N_{i-1}$ and $v$ be a neighbor of $u$ in $N_{i-2}$. Note that $u_lx\in E(G)$ by Claim \ref{su8.3}. If $u_lu\notin E(G)$, then $G[\{x,y,u_l,u\}]\cong K_{1,3}$, a contradiction. Thus we have $u_lu\in E(G)$. Hence we have $G[\{x,y,z,u_l,u,v\}]\cong W$ and $G[\{x,y,u_l,u_{l+1},u,v\}]\cong N$, contradicting Claim \ref{su8.2}.

The proof is complete.
\end{proof}

Now we choose $k$ and $l$ such that\\
(1) $G[u_{-k},u_l]$ is $(u_{-k},u_0,u_l)$-composed with the canonical ordering $u_{-k},u_{-k+1},\ldots,u_l$;\\
(2) there is not a heavy pair in $G[u_{-k},u_l]$;\\
(3) there is not a heavy triangle in $G[u_{-k},u_l]$; and\\
(4) $k+l$ is as large as possible.

By Claim \ref{cl8}, we know one of the following cases occurs:\\
($a$) there exists a vertex $u_{s'}\in [u_{-k+1},u_l]$ such that $u_{-k-1}u_{s'}\notin E(G)$ and $d(u_{-k-1})+d(u_{s'})\geq n$.\\
($b$) there exists a vertex $u_s\in [u_{-k+1},u_l]$ such that $u_{-k-1}u_s\in E(G)$ and each of $\{u_s,u_{-k-1}\}$ is heavy.\\
($c$) there exists a vertex $u_{t'}\in [u_{-k},u_{l-1}]$ such that $u_{l+1}u_{t'}\notin E(G)$ and $d(u_{l+1})+d(u_{t'})\geq n$.\\
($d$) there exists a vertex $u_t\in [u_{-k},u_{l-1}]$ such that $u_{l+1}u_t\in E(G)$ and each of $\{u_{l+1},u_t\}$ is heavy.\\
($e$) $u_{-k-1}u_{l+1}\notin E(G)$ and $d(u_{-k-1})+d(u_{l+1})\geq n$.\\
(\emph{f}) $u_{-k-1}u_{l+1}\in E(G)$ and each of $\{u_{-k-1},u_{l+1}\}$ is heavy.

Hence there exists a vertex $u_i\in [u_{-k+1},u_l]$ such that $d(u_{-k-1})+d(u_i)\geq n$, or there exists a vertex $u_j\in [u_{-k},u_{l-1}]$ such that $d(u_{l+1})+d(u_j)\geq n$, or $d(u_{-k-1})+d(u_{l+1})\geq n$.

\begin{claim}\label{cl9}
$(u_{-k-1},u_l)$ or $(u_{-k},u_{l+1})$ or $(u_{-k-1},u_{l+1})$ is $u_0$-good on $C$.
\end{claim}

\begin{proof}
If there exists a vertex $u_i\in [u_{-k+1},u_l]$ such that $d(u_{-k-1})+d(u_i)\geq n$, then since $G[u_{-k},u_l]$ is $(u_{-k},u_0,u_l)$-composed with the canonical ordering $u_{-k},u_{-k+1},\ldots,u_l$, there exists a $(u_0,u_l)$-path $P$ such that $V(P)=[u_{-k-1},u_l]\setminus\{u_{-k-1}\}$. Moreover, there is a $(u_0u_l,u_iu_{-k})$-pair $D$ such that $V(D)=[u_{-k},u_l]$, then $D'=D+u_{-k}u_{-k-1}$ is a $(u_0u_l,u_iu_{-k-1})$-pair $D'$ such that $V(D')=[u_{-k-1},u_l]$. Therefore $(u_{-k-1},u_l)$ is $u_0$-good on $C$.

If there exists a vertex $u_j\in [u_{-k},u_{l-1}]$ such that $d(u_{l+1})+d(u_j)\geq n$, then we can prove this claim similarly.

Now suppose $d(u_{-k-1})+d(u_{l+1})\geq n$. By Lemma \ref{le4}, there exists a $(u_0,u_l)$-path $P'$ such that $V(P')=[u_{-k},u_l]$ and a $(u_0,u_{-k})$-path $P''$ such that $V(P'')=[u_{-k},u_l]$. Then $P=P'u_lu_{l+1}$ is a $(u_0,u_{l+1})$-path such that $V(P)=[u_{-k},u_{l+1}]$, and $D=P''u_{-k}u_{-k-1}\cup u_{l+1}$ is a $(u_0u_{l+1},u_{l+1}u_{-k-1})$-pair such that $V(D)=[u_{-k-1},u_{l+1}]$. Thus $(u_{-k-1},u_{l+1})$ is $u_0$-good on $C$.

The proof is complete.
\end{proof}

\begin{claim}\label{cl10}
There exists a vertex $v_{-k'}\in V(\overrightarrow{C}[v_1,u_{-k-1}])$ and $v_{l'}\in V(\overleftarrow{C}[v_{-1},u_{l+1}])$ with $(v_{-k'},v_{l'})$ is $v_0$-good.
\end{claim}

\begin{proof}
If $v_{-1}v_1\notin E(G)$, then $d(v_{-1})+d(v_1)\geq n$ by Claim \ref{cl2}. Moreover, $P=v_0v_{-1}$ is a $(v_0v_{-1})$-path and $D=v_0v_1\cup \{v_{-1}\}$ is a $(v_0v_{-1},v_{-1}v_1)$-pair. Hence $(v_{-1},v_1)$ is $v_0$-good.

If $v_{-1}v_1\in E(G)$, then $G[v_{-1},v_1]$ is $(v_{-1},v_0,v_1)$-composed.

Now, set ${r'_1}=r_1-l$ and ${r'_2}=r_2-k$, where $k\leq r_2-2$ and $l\leq r_1-2$ by Claim \ref{cl7}. Similar to Claims \ref{cl7} and \ref{cl8}, we have Claims \ref{su10.1} and \ref{su10.2} as follows.

\begin{subclaim}\label{su10.1}
If $G[v_{-k'},v_{l'}]$ is $(v_{-k'},v_0,v_{l'})$-composed with the canonical ordering $v_{-k'},v_{-k'+1},$\\$\ldots,v_{l'}$, then $k'\leq {r'_1}-1$ and $l'\leq {r'_2}-1$.
\end{subclaim}

\begin{subclaim}\label{su10.2}
If $G[v_{-k'},v_{l'}]$ is $(v_{-k'},v_0,v_{l'})$-composed with the canonical ordering $v_{-k'},v_{-k'+1},$\\$\ldots,v_{l'}$, where $k'\leq {r'_1}-1$ and $l'\leq {r'_2}-1$, and moreover the following two statements hold:\\
($i$) there is not a heavy pair in $G[v_{-k'-1},v_{l'+1}]$,\\
($ii$) there is not a heavy triangle in $G[v_{-k'-1},v_{l'+1}]$,\\
then one of the following is true:\\
(1) $G[v_{-k'-1},v_{l'}]$ is $(v_{-k'-1},v_0,v_{l'})$-composed with the canonical ordering $v_{-k'-1},v_{-k'},\ldots,v_{l'}$,\\
(2) $G[v_{-k'},v_{l'+1}]$ is $(v_{-k'},v_0,v_{l'+1})$-composed with the canonical ordering $v_{-k'},v_{-k'+1},\ldots,v_{l'+1}$,\\
(3) $G[v_{-k'-1},v_{l'+1}]$ is $(v_{-k'-1},v_0,v_{l'+1})$-composed with the canonical ordering $v_{-k'-1},v_{-k'},\ldots,v_{l'+1}$.
\end{subclaim}

Now we choose $k'$ and $l'$ such that\\
(1) $G[v_{-k'},v_{l'}]$ is $(v_{-k'},v_0,v_{l'})$-composed with the canonical ordering $v_{-k'},v_{-k'+1},\ldots,v_{l'}$;\\
(2) there is not a heavy pair in $G[v_{-k'},v_{l'}]$;\\
(3) there is not a heavy triangle in $G[v_{-k'},v_{l'}]$; and\\
(4) $k'+l'$ is as large as possible.

Similar to Claim \ref{cl9}, $(v_{-k'-1},v_{l'})$ or $(v_{-k'},v_{l'+1})$ or $(v_{-k'-1},v_{l'+1})$ is $v_0$-good on $C$.
\end{proof}

By Claims \ref{cl9} and \ref{cl10}, there exists a cycle which contains all the vertices in $V(C)\cup V(P)$ by Lemma \ref{le5}, a contradiction.

The proof is complete.
\hfill $\Box$

\section{Concluding Remarks}

It is known that Faudree and Gould \cite{Faudree_Gould} extended Bedrossian's result to 2-connected graphs on $n\geq 10$ vertices.

\begin{theorem}[Faudree and Gould \cite{Faudree_Gould}]\label{th15}
Let $R$ and $S$ be connected graphs other than $P_3$ and let $G$ be a 2-connected graph on $n\geq 10$ vertices. Then $G$ being $\{R,S\}$-free implies $G$ is Hamiltonian if and only if (up to symmetry) $R=K_{1,3}$ and $S=P_4,P_5,P_6,C_3,Z_1,Z_2,Z_3,B,$\\$N$ or $W$.
\end{theorem}

Chen et al. \cite{Chen_Wei_Zhang_5} showed every 2-connected $\{K_{1,3},Z_3\}$-\emph{f}-heavy graph on $n\geq 10$ vertices is Hamiltonian.

\begin{theorem}[Chen, Wei and Zhang \cite{Chen_Wei_Zhang_5}]\label{th16}
Let $G$ be a 2-connected graph on $n\geq 10$ vertices. If $G$ is $\{K_{1,3},Z_3\}$-f-heavy, then $G$ is Hamiltonian.
\end{theorem}

Together with Theorems \ref{th12}, \ref{th16} and Remark \ref{re1}, we have the following result which extends Theorem \ref{th15}.

\begin{theorem}
Let $R$ and $S$ be connected graphs other than $P_3$ and let $G$ be a 2-connected graph on $n\geq 10$ vertices. Then $G$ being $\{R,S\}$-f-heavy implies $G$ is Hamiltonian if and only if (up to symmetry) $R=K_{1,3}$ and $S=P_4,P_5,P_6,Z_1,Z_2,Z_3,B,N$ or $W$.
\end{theorem}

Li et al. \cite{Li_Ryjacek_Wang_Zhang} also constructed a class of 2-connected graphs on $n\geq 10$ vertices which are $\{K_{1,3},Z_3\}$-\emph{o}-heavy but not Hamiltonian. Thus it is natural to pose the following problem.

\begin{problem}
Is every 2-connected claw-\emph{o}-heavy and $Z_3$-\emph{f}-heavy graph on $n\geq 10$ vertices Hamiltonian?
\end{problem}

\section*{Acknowledgement}
The authors are grateful to two anonymous referees for helpful comments on an earlier version of this article.

\end{document}